\documentclass[a4paper]{amsart}
\usepackage {amsmath, amscd}
\usepackage {amssymb}
\usepackage{bm,color}
\usepackage{hyperref,cite}
\usepackage{color}

\numberwithin{equation}{section}

\def\<{\langle}
\def\>{\rangle}

\def\DD{{\mathcal D}}
\def\EE{{\mathcal E}}

\def\HH{{\mathcal H}}

\def\KK{{\mathcal K}}
\def\LL{{\mathcal L}}

\def\bbN{\mathbb{N}}

\def\bbD{\mathbb{D}}
\def\bbT{\mathbb{T}}

\def\FFF{\mathfrak{F}}

\newtheorem{lemma}{Lemma}[section]
\newtheorem{theorem}[lemma]{Theorem}
\newtheorem{corollary}[lemma]{Corollary}
\newtheorem{conjecture}[lemma]{Conjecture}

\theoremstyle{definition}
\newtheorem{remark}[lemma]{Remark}

\title{Reducing subspaces of $ C_{00} $ contractions}

\author[Benhida]{Chafiq Benhida} 
\address{Laboratoire Paul Painlev\'e, Universit\'e Lille 1, 59 655 Villeneuve d'Ascq C\'edex }
 \email{chafiq.benhida@univ-lille.fr}
 
\author[Fricain]{Emmanuel Fricain}
\address{Laboratoire Paul Painlev\'e, Universit\'e Lille 1, 59 655 Villeneuve d'Ascq C\'edex }
\email{emmanuel.fricain@univ-lille.fr}

\author[Timotin]{Dan Timotin}
 \address{Simion Stoilov Institute of Mathematics of the Romanian Academy, PO Box 1--764, Bucharest 014700, Romania}
 \email{Dan.Timotin@imar.ro}


\thanks{This work was partially supported by Labex CEMPI (ANR-11-LABX-0007).}
\begin{document}

\begin{abstract}
Using the Sz.-Nagy--Foias theory of contractions, we obtain general results about reducibility for a class of completely nonunitary contractions.  These are applied to certain truncated Toeplitz operators, previously considered by Li--Yang--Lu and Gu. In particular, a negative answer is given to a conjecture stated by the latter.
\end{abstract}

\maketitle

\section{Introduction}

We will denote by $ L^2 $   the Lebesgue space $ L^2(\bbT, dm) $, where $ dm $ is normalized Lebesgue measure. The subspace of functions whose negative Fourier coefficients are zero is denoted by $ H^2 $; it is identified with the space of analytic functions in the unit disc with square summable Taylor coefficients. An inner function is an element of $ H^2 $ whose values have modulus 1 almost everywhere on $ \bbT $.

If $ \theta $ is an inner function, then the space $ K_\theta=H^2\ominus \theta H^2 $ is usually called a model space; it has been the focus of much research, in function theory in the unit disc as well as in operator theory (see, for instance,\cite{S1, Nik}; or~\cite{GMR} for a more recent account). In particular, in the last two decades several papers discuss the so-called \emph{truncated Toeplitz operators}, introduced in~\cite{S2}, which are compressions to $ K_\theta $ of multiplication operators on $ L^2 $.

Originating with work in~\cite{DF}, the question of reducibility of a certain class of truncated Toeplitz operators has been recently investigated in papers by Yi, Yang, and Lu~\cite{LYL0, LYL} and Gu~\cite{Gu}. Besides certain remarkable results, they also contain intriguing questions that have not yet found their solution. 

The current paper has several purposes. First, we put the problem of reducibility of the truncated Toeplitz operators in a larger context, that of the Sz.-Nagy--Foias theory of completely nonunitary contractions~\cite{NF}, and show that some results in the above quoted papers may be generalized or given more transparent proofs. Secondly, we answer in the negative a conjecture stated in~\cite{Gu} and prove a statement that replaces it.

The plan of the paper is the following. After presenting in the next section the elements of Sz.-Nagy theory that interest us, we obtain in Section~\ref{se:reducibility general} some general results about reducibility for completely nonunitary contractions. These results are applied in Section~\ref{se:a class} to a certain class of truncated Toeplitz operators. The connection to~\cite{LYL} is achieved in Section~\ref{se:part case}, while the relation to~\cite{Gu} is the content of the last section.

	\section{Sz.-Nagy--Foias dilation theory}\label{se:nagy-foias}
The general reference for   this section is the monograph~\cite{NF}, in particular chapters I, II, and VI.

\subsection{Minimal isometric dilation}

If $ \HH $ is a Hilbert space and $ \HH_1 $ is a closed subspace, we will denote by $ P_{\HH_1} $ the orthogonal projection onto $ \HH_1 $.

A closed subspace $M$ of $\HH$ is said to be \emph{reducing} for an operator $T$ if both $M$ and $M^\perp$ are invariant with respect to $T$.
A \emph{completely nonunitary contraction} $ T\in \LL(\HH) $ is a linear operator that satisfies $ \|T\|\le 1 $, and there is no reducing subspace of $ T $ on which it is unitary.  The defect of $T$ is the operator $ D_T=(I-T^*T)^{1/2} $, and the defect space is $\DD_T=\overline{D_T\HH} $. 

We write   $T\in C_{\cdot 0}$ if $T^*{}^n$ tends strongly to $0$, and  $ T\in C_{00} $ if $T$ and $T^*$ are in $C_{\cdot 0}$, that is $ T^n $ and $ T^*{}^n $ both tend strongly to $0$. If $T\in C_{00}$, then it can be shown that   $ \dim\DD_{T}=\dim\DD_{T^*} $. The subclass of $ C_{00} $ for which this dimension is finite and equal to~$ N $ is denoted by $ C_0(N) $.  We will mostly be interested by contractions in the class $ C_{00} $.

An isometric dilation of $T$ is an isometric operator $ V\in \LL(\KK) $, with $ \KK\supset \HH $, such that $ P_\HH V^n|\HH = T^n $ for any $ n\in\bbN $. Note that if $T=P_\HH V|\HH$ and $VH^\perp\subset H^\perp$, then $V$ is a dilation. An isometric dilation $V\in \LL(\KK) $ is called minimal if $ \KK=\bigvee_{n=0}^\infty V^n\HH $.
 This is uniquely defined, modulo a unitary isomorphism commuting with the dilations;
 in~\cite{NF} there is a precise description of its geometric structure. This becomes simpler  for contractions in $ C_{\cdot0} $; since this is the only case we are interested in, we will describe the minimal isometric dilation in this case.

We will say that a subspace $ X\subset \KK $ is \emph{wandering} for $ V $ if $ V^nX\perp V^mX $ for any $ n\not=m $, and in this case we will denote $ M_+(X):=\bigoplus_{n=0}^\infty V^nX $. Note that $ M_+(X) $ is invariant with respect to $ V $. 

\begin{lemma}\label{le:structure of minimal dilation} If $ T $ is a completely nonunitary contraction and $ V $ is its minimal isometric dilation, then 
 $ T\in C_{\cdot0} $ if and only if there exist wandering subspaces $ L, L_* \subset \KK $ for $ V $, with $ \dim L=\dim \DD_T$ and $\dim L_*=\dim \DD_{T^*}$, such that 
\begin{equation}\label{eq:structure of minimal dilation}
	\KK= M_+(L_*)=\HH\oplus M_+(L).
\end{equation}

In this case,
 the operators
\begin{equation}\label{eq:identification D_T L}
	\phi:D_T x\mapsto (V-T)x, \quad \phi_*:D_{T^*} x\mapsto x-VT^* x
\end{equation}
extend to unitary operators $ \DD_T\to L $ and $ \DD_{T^*} \to L_*$. 
	
\end{lemma}

%

\subsection{Analytic vector valued functions}

If $ \EE $ is a Hilbert space, then $ H^2(\EE) $ is the Hilbert space of $ \EE $-valued analytic functions in $ \bbD $ with the norms of the Taylor coefficients square summable. As in the scalar case, these functions have strong radial limits almost everywhere on $ \bbT $, and so may be identified with their boundary values, defined on $ \bbT $.

Denote by $ \bm{T}^\EE_z $ multiplication by $ z $ acting on  $H^2(\EE)$; it is an isometric operator. 
If $ \omega:\EE\to \EE' $ is unitary, then the notation $ \tilde{\omega} $ will indicate the unique unitary extension $ \tilde{\omega}:H^2(\EE)\to H^2(\EE') $ such that
$\tilde{\omega} \bm{T}^\EE_z= \bm{T}^{\EE'}_z \tilde{\omega}  $.  

Suppose $ X\subset \KK $ is wandering for the isometry $ V\in\LL(\KK) $. 
Then
the map $  \FFF_X $, defined by 
\begin{equation}\label{eq:Fourier representation}
	\FFF_X(\sum_{n=0}^\infty V^n x_n)=\sum_{n=0}^\infty \lambda^n x_n,
\end{equation} 
is unitary from $ M_+(X) $ to $ H^2(X) $.

Another class of functions that we have to consider take as values operators between two Hilbert spaces $ \EE, \EE_* $. More precisely, we will be interested in  \emph{contractive analytic functions}; that is, functions   $ \Theta:\bbD\to  \LL(\EE, \EE^*) $, which satisfy $ \|\Theta(z) \|\le 1$ for all $ z\in\bbD $. As in the scalar case, $ \Theta $ has boundary values $ \Theta(e^{it}) $ almost everywhere on~$ \bbT $.

A contractive analytic function is called \emph{pure} if $\| \Theta(0)x\|<\| 
x\| $ for any $x\in \EE$, $x\neq 0$.
Any contractive analytic function admits a decomposition in a direct sum $\Theta=\Theta_p\oplus \Theta_u   $, where $\Theta_p$ is pure and $\Theta_u$ is a constant unitary operator; then $\Theta_p$ is called the \emph{pure part} of $\Theta$.
A contractive analytic function will be called \emph{bi-inner} if $ \Theta(e^{it}) $ is almost everywhere unitary. (We prefer this shorter word rather than call them inner and *-inner).

The appropriate equivalence relation for contractive analytic functions is that of coincidence: two analytic functions $ \Theta:\bbD\to \LL(\EE, \EE_*) $,  $ \Theta':\bbD\to \LL(\EE', \EE'_*) $ are said to \emph{coincide} if there exist unitary operators $ \omega:\EE\to \EE' $, $ \omega_*:\EE_*'\to \EE_*' $, such that $ \Theta'(\lambda)\omega=\omega_*\Theta(\lambda) $ for all $ \lambda\in\bbD $. 

\subsection{Functional model and characteristic function}

The model theory of Sz.-Nagy and Foias associates to any completely nonunitary contraction $ T $ a pure contractive analytic function $ \Theta_T(z) $, with values in $ \LL(\DD_T, \DD_{T^*}) $, defined by the formula
\begin{equation}\label{eq:charactersitic function}
	\Theta_T(z)=-T+zD_{T^*}(I-zT^*)^{-1}D_T|\DD_T.
\end{equation}
A functional model space and an associated model operator are constructed by means of $ \Theta_T $, and one can prove that $ T $ is unitarily equivalent to this model operator. 

As we will be interested only in $ C_{00} $ contractions, we will describe the model only in this case, in which it takes a significantly simpler form.  The reason is that $T\in C_{00} $ is equivalent to $ \Theta_T $  bi-inner.
The functional model associated to a bi-inner contractive analytic function $ \Theta:\bbD\to  \LL(\EE, \EE^*) $ is defined as follows:	
the \emph{model space} is
\begin{equation}\label{eq:model space}
	\HH_\Theta= H^2(\EE_*) \ominus \Theta   H^2(\EE),
\end{equation}
while the \emph{model operator} $ \bm{S}_\Theta $ is the compression to $ \HH_\Theta $ of $ \bm{T}_z^{\EE_*} $. If $ \Theta $ is pure, then $ \bm{T}_z^{\EE_*}  $ is precisely a minimal unitary dilation of $ \bm{S}_\Theta $.

	Note that \eqref{eq:model space} shows that $ \bm{S}_\Theta $ satisfies 
	the assumptions of Lemma~\ref{le:structure of minimal dilation} with $ L=\Theta \EE $, $ L_*=\EE_* $. In particular, 
	\begin{equation}\label{eq:defects for model}
		\dim \DD_{\bm{S}_\Theta}=\dim \EE,\qquad
			\dim \DD_{\bm{S}^*_\Theta}=\dim \EE_*.	
	\end{equation}

Suppose $ \Theta:\bbD\to \LL(\EE, \EE_*) $ and  $ \Theta':\bbD\to \LL(\EE', \EE'_*) $ coincide, by means of the  operators $ \omega:\EE\to \EE' $, $ \omega_*:\EE_*'\to \EE_*' $. Then the unitary $\tilde{\omega}_*: H^2(\EE_*) \to H^2(\EE'_*) $ satisfies $ \tilde{\omega}(\HH_\Theta)=\HH_{\Theta'} $ and 
\[
\tilde{\omega}_*\bm{S}_\Theta=\bm{S}_{\Theta'}\tilde{\omega}_*.
\]

Returning now to the contraction $ T $ and its characteristic function, the next lemma is a particular case of one of the basic results in~\cite[Chapter VI]{NF}.

\begin{lemma}\label{le:model for C.0}
If   $ T\in C_{00}$, then the formula
 \eqref{eq:charactersitic function}
defines a bi-inner pure   analytic 
function with values in $ \LL(\DD_T, \DD_{T^*}) $, and $ T $
is unitarily equivalent to $ \bm{S}_{\Theta_T} $. $ \bm{S}_{\Theta_T} $ is called the \emph{functional model} of $ T $.	
\end{lemma}

There is a  relation between the functional model and the geometrical structure of a minimal unitary dilation given by~\eqref{eq:structure of minimal dilation}, as shown by the next result. 

\begin{lemma}\label{le:relation minimal characteristic}
	Suppose $ T\in C_{00} $, $ V\in\LL(\KK) $ is a minimal isometric dilation of $ T $, and $ L,L_* $ are wandering subspaces for $ V $ satisfying~\eqref{eq:structure of minimal dilation}. 	Extend $ \phi, \phi_* $   in~\eqref{eq:identification D_T L}  to  unitary operators $ \tilde\phi :H^2(\DD_{T})\to H^2(L) $,  $ \tilde\phi_*:H^2(\DD_{T^*})\to H^2(L_*) $, and  define  $ \Omega=\FFF_{L_*}^*\Phi_*$.  
	
	\begin{itemize}
		\item[\rm (i)] The map $	\FFF_{L_*}^*\Phi_*\Theta_T \Phi^* \FFF_L  $ is the inclusion of $ M_+(L) $ into $ M_+(L_*) $.
	\item[\rm (ii)]	We have 
		\begin{equation}\label{eq:fourier minimal dilation}
			\begin{split}
				\Omega \HH_{\Theta_T}&= \HH, \quad \Omega (\DD_{T^*})=L_*,
				\quad \Omega\Theta_T (\DD_{T})=L,\\ \Omega \bm{T}_z^{\DD_{T^*}}&=V\Omega, \quad\Omega \bm{S}_{\Theta_T}= T\Omega.
			\end{split}
		\end{equation}
	
	\item[\rm (iii)] If $ \Theta=\phi_* \Theta_T \phi^* $ is  written
	$ \Theta(\lambda)=\sum_{n=0}^\infty \lambda^n \Theta_n $ (with $ \Theta_n:L\to L_* $), then
	\begin{equation}\label{eq:formula for Theta_n}
		\Theta_n=P_{L_*}(V^*)^nJ,
	\end{equation}
	\end{itemize}
	where $ J $ denotes the embedding of $ L $ into $ M_+(L_*) $.
		
\end{lemma}

\section{Reducibility}\label{se:reducibility general}

In the sequel of the paper we will   be interested by reducibility of certain contractions. 
Fortunately, this can be easily characterized through characteristic functions.

\begin{lemma}\label{le:general reducibility}
	Suppose $ T\in C_{00} $ has characteristic function $ \Theta_T:\DD_T\to \DD_{T^*} $. Then the following are equivalent.
	
	\begin{itemize}
		\item[\rm(i)] $ T=T_1\oplus T_2 $.
		
		\item[\rm(ii)] There exist nontrivial orthogonal decompositions $ \DD_T=E^1\oplus E^2 $, $ \DD_{T^*}=E^1_*\oplus E^2_* $  which diagonalize $ \Theta_T(\lambda) $ for all $ \lambda\in\bbD $; that is,
		\begin{equation}\label{eq:decomposition of Theta}
			\Theta_T(\lambda)=
			\begin{pmatrix}
				\Theta_1(\lambda)&0\\0& \Theta_2(\lambda)
			\end{pmatrix}.
		\end{equation}
		
	\end{itemize}
	In this case $ \dim\DD_{T_i}=\dim E^i = \dim \DD_{T_i^*}=\dim E^i_* $, and $ \Theta_{T_i} $ coincides with $ \Theta_i $.

\end{lemma}

\begin{proof}
	If $ T=T_1\oplus T_2 $, then $ \DD_T=\DD_{T_1}\oplus \DD_{T_2} $, $ \DD_{T^*}=\DD_{T_1^*}\oplus \DD_{T_2^*} $, and formula~\eqref{eq:charactersitic function} splits according to these decompositions into $ \Theta_T(\lambda)=\Theta_{T_1}(\lambda)\oplus \Theta_{T_2}(\lambda) $.  So~\eqref{eq:decomposition of Theta} is valid, taking $ E^i=\DD_{T_i} $, $ E^i_*=\DD_{T_i^*} $.
	
	Conversely, if $\Theta(\lambda):=\Theta_T(\lambda)=\Theta_1(\lambda)\oplus \Theta_2(\lambda) $, then, according to~\eqref{eq:model space}, $ \HH_\Theta= \HH_{\Theta_1}\oplus \HH_{\Theta_2} $, and $ \HH_{\Theta_1}, \HH_{\Theta_2} $ are invariant with respect to $ \bm S_\Theta $. Since this last operator is unitarily equivalent to $ T $, $ T $ is also reducible. Moreover, $ \bm{S}_\Theta|\HH_{\Theta_i} $ is unitarily equivalent to $ \bm{S}_{\Theta_i} $, and the equality of the dimensions follows from~\eqref{eq:defects for model}.
\end{proof}

\begin{corollary}\label{co:reducibility for C_00}
	Suppose $ T\in C_{00} $. Then
  $ T $ is reducible if and only if there exist nontrivial subspaces $ E\subset \DD_T $, $ E_*\subset \DD_{T^*} $, such that  $\Theta_T(e^{it})E= E_* $ for almost all~$t$.

\end{corollary}

\begin{proof}
	If nontrivial subspaces as assumed exist, then, since $ \Theta_T(e^{it}) $ is unitary almost everywhere, we also have $\Theta_T(e^{it})E^\perp= E_*^\perp $ for almost all $ t $. The decompositions $ \DD_T=E\oplus E^\perp $, $ \DD_{T^*} =E_*\oplus E_*^\perp$ satisfy then~\eqref{eq:decomposition of Theta}.  
	\end{proof}

The following is a geometrical reformulation of Corollary~\ref{co:reducibility for C_00} in terms of the spaces $ L, L_* $ appearing in an arbitrary minimal isometric dilation of~$ T $.

\begin{corollary}\label{co:geometrical reformulation}
	Suppose $ T\in C_{00} $ and $ V\in\LL(\KK) $ is a minimal dilation of $ T $, such that~\eqref{eq:structure of minimal dilation} is valid for $ L,L_* $ wandering subspaces for~$ V $. 
	Let $ d $ be a finite positive integer or $ \infty $.
	Then:
	\begin{itemize}
		\item[\rm (i)] 	If $ T $ has a nontrivial reducing subspace such that the restriction has $d$-dimensional defects, then there exist nontrivial subspaces $ L^1\subset L $, $ L^1_*\subset L_* $, both of dimension $ d $, such that 
		\begin{equation}\label{eq:L_1 subset M_+(L^1_*)}
			L^1\subset M_+(L^1_*).
		\end{equation}
	
	\item[\rm(ii)] The converse also holds if $ d<\infty $.
	\end{itemize}

\end{corollary}

\begin{proof} (i). 
	Suppose $ T $ has a reducing subspace with defect  of dimension $ d $. We apply Lemma~\ref{le:general reducibility}, which gives decomposition~\eqref{eq:decomposition of Theta}, where $ \Theta_i(\lambda):E^i\to E^i_* $, and  $ \dim E^1=\dim E^1_* =d $.  
	So $ \Theta_1 H^2(E^1)\subset H^2(E^1_*) $; in particular, if we look at $E^1  $ as the constant functions in $ H^2(E^1) $, we have
	\begin{equation}\label{eq:22}
		\Theta_1 E^1\subset H^2(E^1_*).
	\end{equation}

	Denote then $ L^1=\phi E^1 $ and $ L^1_*=\phi_* E^1_* $ ($ \phi, \phi_* $ in~\eqref{eq:identification D_T L}). We consider  
	the unitary operator $ \Omega $ from Lemma~\ref{le:relation minimal characteristic}. Formulas~\eqref{eq:fourier minimal dilation} yield also $ \Omega E^1_*=L^1_* $, $ \Omega \Theta_1(E^1)= L^1$, and $ \Omega(H^2(E^1_*))=M_+(L^1_*) $.
	Therefore~\eqref{eq:22} implies~\eqref{eq:L_1 subset M_+(L^1_*)}.
	
	(ii) Conversely, suppose we have the required spaces satisfying~\eqref{eq:L_1 subset M_+(L^1_*)}; therefore $ M_+(L^1)\subset M_+(L^1_*) $. 
	Define $ \Theta'(\lambda)=\phi_*\Theta_T(\lambda)\phi^* :L\to L_*$.
	By using $ \FFF_{L_*} $, we obtain $ \Theta' H^2(L^1)\subset H^2(L^1_*) $, which means that $ \Theta'(e^{it})L^1\subset L^1_* $ almost everywhere. Since $\dim L^1=\dim L^1_*=d<\infty$, we have in fact $ \Theta'(e^{it})L^1=L^1_*$ almost everywhere. As in the proof of Corollary~\ref{co:reducibility for C_00}, it follows that $ \Theta'(e^{it})L^1{}^\perp\subset L^1_*{}^\perp $ almost everywhere, whence we may obtain a decomposition similar to~\eqref{eq:decomposition of Theta}. This implies the reducibility of~$ \Theta' $, and thus the reducibility of $\Theta_T$ and of $T$.		
\end{proof}

%
%
%

In particular, we   obtain a nice result if we consider reducing subspaces with defects of dimension~1.

\begin{corollary}\label{co:case dim of reducing=1}
	An operator $ T\in C_{00} $ has a reducing subspace with defects of dimension~1 if and only if  there exists $ y\in L $, $ y_*\in L_* $, $y,y'\neq 0$, and a scalar inner function $ u $, such that $ y=u(V)y_* $. In this case the characteristic function of the reduced operator is precisely~$ u $.
\end{corollary}

\begin{proof}
	By Corollary~\ref{co:geometrical reformulation} applied to $ d=1 $, the existence of a reducing subspace with defects of dimension~1 is equivalent to the existence of  
	elements of norm 1 $y\in L,  y_*\in L_* $, such that $ y\in M_+(y_*) $. The Fourier representation $ \FFF_{y_*} $ maps   $ M_+(y_*) $ onto $ H^2 $; more precisely, from~\eqref{eq:Fourier representation} it follows that $\FFF_{y_*} (f(V)y_*)=f  $. In particular, $ y $ is a wandering vector for $ V $, which implies that $u:= \FFF_{y_*} y $ is an inner function.
	
	If we denote by $ \HH_1 $ the reducing subspace of dimension 1 obtained, then   have $ \HH_1=M^+(y_*)\ominus M^+(y) $. Through the Fourier representation $ \FFF_{y_*} $, this becomes $ H^2\ominus uH^2 $. By comparing with the general formula for the functional model, we see that the characteristic function of the reduced operator is~$ u $. 
	\end{proof}

\begin{remark}

	Part of the results in this section may be extended to more general contractions. Thus Lemma~\ref{le:general reducibility} is true for a general completely nonunitary contraction; we have then to use in the proof the more complicated general form of the functional model associated to~$ T $. Appropriately modified versions of Corollaries~\ref{co:reducibility for C_00} and~\ref{co:geometrical reformulation} can also be stated. However,  since the statements are less neat, we have preferred to restrict ourselves to the case $ T\in C_{00} $, which will be used in the applications in the sequel of the paper.
\end{remark}

\section{A class of contractions}\label{se:a class}
 In the rest of the paper we will work in the Hardy space $ H^2 $, applying  the above results to a particular class of contractions. 
 By $ T_\varphi $ we will denote the usual Toeplitz operator on $ H^2 $, that is, the compression of the operator of multiplication by $\varphi$ on $H^2$. Recall here that, for a scalar inner function $K_\theta=H^2\ominus \theta H^2=\HH_\theta$ (see \eqref{eq:model space} with $\EE=\EE_*=\mathbb C$).

Let then $ \theta, B $ be two  scalar inner functions that satisfy the basic assumption
\begin{equation}\label{eq:assumption dim2}
	\ker T_{\theta \overline B}=\{0\}.
\end{equation}
Note that $f\in\ker T_{\theta \overline B}$ if and only if $\theta f\in \ker T_{\overline B}=K_B$, whence \eqref{eq:assumption dim2} is equivalent to $\theta H^2\cap K_B=\{0\}$.

We will consider the operator $ A^\theta_B\in\LL(K_\theta) $, defined by 
	\begin{equation}\label{eq:definition of A}
	A^\theta_B=P_{K_\theta}T_B|K_\theta.
\end{equation}
 
The operator $A^\theta_B$ is usually called the \emph{truncated Toeplitz operator} on $ K_\theta $ with symbol $ B $. It is known~\cite{S2} that truncated Toeplitz operators are \emph{complex symmetric}; that is, there exist a complex conjugation $ C_\theta $ on~$ K_\theta $ such that 
\begin{equation}\label{eq:complex symmetric}
		(A^\theta_B)^*=C_\theta	A^\theta_B C_\theta.
\end{equation}

The next theorem identifies concretely a minimal isometric dilation of $ A^\theta_B $; it is a generalization of~\cite[Lemma 3.1]{LYL}.

	\begin{theorem}\label{pr:T_B is a minimal dilation}
	 Let $B$ and $\theta$ two inner functions satisfying \eqref{eq:assumption dim2}.	The operator $ T_B\in\LL(H^2) $ is a minimal isometric dilation of $ A^\theta_B $. For this minimal isometric dilation we have 
	\begin{equation}\label{eq:L,L_* in the particular case}
		L=\theta K_B, \quad L_*=K_B.
	\end{equation}
\end{theorem}

\begin{proof}
	$ T_B $ is an isometry on $ H^2 $,  and $T_B(K_\theta^\perp)=T_B(\theta H^2)\subset \theta H^2=K_\theta^\perp$. 
	Thus it follows from~\eqref{eq:definition of A} that $T_B$ is a dilation of~$ A^\theta_B $.
	
	We show now by induction according to $ n $ that
	\begin{equation}\label{eq:induction TB^n}
		K_\theta + BK_\theta+\dots+ B^nK_\theta= K_{B^n\theta}.
	\end{equation}
	Equality~\eqref{eq:induction TB^n} is obviously true for $ n=0 $. Suppose that it is true up to $ n-1 $. We are left then to prove that
	\begin{equation}\label{eq:ind1}
		K_{B^{n-1}\theta}+B^n K_\theta= K_{B^n\theta}.
	\end{equation}	
	
	It is immediate from the definitions that the left hand side is contained in the right hand side. On the 
	other hand, we have
	\[
	K_{B^n\theta}=
	K_{B^{n-1}\theta}\oplus B^{n-1}\theta K_B
	=B^n K_\theta\oplus K_{B^n}.
	\]
	If $ f\in 	K_{B^n\theta} $ is orthogonal to $ 	K_{B^{n-1}\theta} $ as well as to $ B^n K_\theta $, it follows that 
	\[
	f\in (\theta B^{n-1}K_B)\cap K_{B^n}.
	\]
	So $ f=\theta B^{n-1}g $ with $ g\in K_B $; and also $ f\perp B^n H^2 $, which means $ \theta g\perp BH^2 $, or $ \theta g\in K_B $.   It follows that $0=T_{\overline B}(\theta g)=T_{\theta \overline B}g$. By~\eqref{eq:assumption dim2}, this implies $ g=0 $, whence $ f=0 $.
	
	Since 
	\[
	\big(\bigvee_n  K_{B^n\theta}\big)^\perp=\bigcap_n B^n\theta H^2=\{0\},
	\]
	it follows that
	\begin{equation}\label{eq:TB^nKtheta}
		H^2=\bigvee_{n=0}^\infty T_B^n K_\theta.
	\end{equation}
	Therefore $ T_B $ is a \emph{minimal} isometric dilation of~$ A_B^\theta $.
	
	Then
	\begin{equation}\label{eq:decomposition with K_B}
		H^2=\bigoplus_{n=0}^\infty B^n K_B = \bigoplus_{n=0}^\infty T_B^n K_B= M_+(K_B),
	\end{equation}
	whence $ L_*=K_B $.
	
	On the other hand, we have
	\begin{equation}\label{eq:KB Ktheta}
		K_{B\theta}=K_\theta\oplus \theta K_B=K_B\oplus B K_\theta,
	\end{equation}
	Therefore
	\begin{equation}\label{eq:structure in the case of B}
		H^2=\bigoplus_{n=0}^\infty B^n K_B=K_\theta \oplus \theta H^2
		=K_\theta\oplus \bigoplus_{n=0}^\infty T_B^n\theta  K_B= K_\theta\oplus M_+(\theta K_B),
	\end{equation}
	whence $ L=\theta K_B $.
\end{proof}

\begin{corollary}\label{co:A theta B is C00}
	With the above assumptions, $A_B^\theta$ is in $C_{00}$.
\end{corollary}

\begin{proof}
	In view of equation~\eqref{eq:structure in the case of B}, it follows from Lemma~\ref{le:structure of minimal dilation} that $A_B^\theta$ is in $C_{\cdot0}$. On the other hand, it follows from~\eqref{eq:complex symmetric} that
	\[
		((A^\theta_B)^*)^n=C_\theta	(A^\theta_B)^n C_\theta,
	\]
	whence $A_B^\theta$ is also in $C_{0\cdot}$.
\end{proof}

Using the identification of a minimal unitary dilation in Theorem~\ref{pr:T_B is a minimal dilation} we may compute the characteristic function of~$ A^\theta_B $. The next theorem generalizes~\cite[Theorem 2.4]{Gu} (see Section~\ref{se:monomials} below).

\begin{theorem}\label{th:identification of Phi}
	  Let $B$ and $\theta$ two inner functions satisfying \eqref{eq:assumption dim2}. The characteristic function of $ A^\theta_B $ is $ \Phi:\bbD\to\LL(K_B) $ defined by
	\begin{equation}\label{eq:identification of Phi}
		\Phi(\lambda)=A^B_{\frac{\theta}{1-\lambda \overline{B}}}.
	\end{equation}
\end{theorem}

\begin{proof}
	We have identified in Theorem~\ref{pr:T_B is a minimal dilation} $ L, L_* $ with $ \theta K_B,  K_B $ respectively. We intend to apply Lemma~\ref{le:relation minimal characteristic} (iii). Since we want to consider the characteristic function of $ A $ as an analytic function with values in $ \LL( K_B) $, the embedding $ J $ is precisely multiplication by $ \theta $. 
	Then, if $ \Phi(\lambda)=\sum_{n=0}^\infty \lambda^n\Phi_n $,~\eqref{eq:formula for Theta_n} yields
	\[
	\Phi_n f=P_{K_B} \overline{B}^n \theta f
	\]
	for $ f\in K_B $. Thus $ \Phi_n= A^B_{\theta \overline{B}^n} $. Therefore
	\[
	\Phi(\lambda)=\sum_{n=0}^\infty \lambda^n A^B_{\theta \overline{B}^n}
	= A^B_{\theta\sum_{n=0}^\infty\lambda^n \overline{B}^n}
	=A^B_{\frac{\theta}{1-\lambda \overline{B}}}.\qedhere
	\]
\end{proof}

 	We may also obtain a more precise form of Corollary~\ref{co:case dim of reducing=1}.
 	
\begin{corollary}\label{cor-reducing-general-case}
Let $B$ and $\theta$ two inner functions satisfying \eqref{eq:assumption dim2}. Then the following assertions are equivalent:
\begin{itemize}
	\item[\rm (i)]  The operator $ A^\theta_B $ has a reducing subspace such that the restriction has one-dimensional defects.
	
	\item[\rm (ii)] There exist $ u $ inner  and $h_1,h_2\in K_B $, $h_1,h_2\neq 0$, such that
		\begin{equation}\label{eq:main one-dim}
		\theta=\frac{h_2}{h_1}(u\circ B),
	\end{equation}

\item[\rm(iii)] There exist $ u, v_1, v_2 $ inner, with 
	\[
\ker T_{v_1\bar B}\cap 	\ker T_{v_2\bar B}\not=\{0\}, 
\]
such that
	\begin{equation}\label{eq:main one-dim2}
	\theta=\frac{v_2}{v_1}(u\circ B).
\end{equation}
\end{itemize}

\end{corollary}
\begin{proof}
The equivalence of (i) and (ii) follows by applying in this case Corollary~\ref{co:case dim of reducing=1}. We have $ L_*=K_B $, $ L=\theta K_B $, and so the existence of the required reducing subspace is equivalent to the existence of $ h_1, h_2\in K_B $, $h_1,h_2\neq 0$ and an inner function $ u $, such that $ \theta h_1= u(V)h_2 $. Since   $ V=T_B $, $ u(V) $ is multiplication by $ u\circ B $, and we have
		\begin{equation}\label{eq:theta h_1=h_2 u}
			\theta h_1=h_2(u\circ B).
		\end{equation}
	
	If (ii) is true, then we must have $ h_i=v_i g $ for some inner functions $ v_1, v_2 $ and $ g $ outer, so~\eqref{eq:main one-dim2} is true. Note that, if $ v $ is an analytic and bounded function in $\mathbb D$, then
	\begin{equation}\label{eq:small observation}
			vh\in K_B \Leftrightarrow h\in\ker T_{v\bar B}.
	\end{equation}
	So $ v_1 g, v_2 g\in K_B $ is equivalent to $ g\in \ker T_{v_1\bar B}\cap 	\ker T_{v_2\bar B} $.
	
	The implication (iii)$ \implies $(ii) follows easily by reversing the steps. 
	\end{proof}

 Note that the function $u$ in (ii) and (iii) of the previous corollary is non constant because otherwise $\theta h_1\in K_B$, and thus 
 $h_1\in\ker T_{\theta\overline B}$ which contradicts hypothesis \eqref{eq:assumption dim2}.

 	\section{A particular case}\label{se:part case}

 
 Let us consider now the particular case when $ B $ is a finite Blaschke product. 
  Denote $ \phi_\alpha(z)=(z-\alpha)/(1-\bar\alpha z) $.
  If $ B $ has roots (counting with multiplicities) $ w_1, \dots, w_k $, it is known that 
 \begin{equation}\label{eq:elements in K_B}
 	K_B=\Big\{ \frac{p(z)}{\prod_{i=1}^k (1-\bar{w_i}z)} : p\text{ polynomial of degree }\le k-1   \Big\}.
 \end{equation}

 In this case condition~\eqref{eq:assumption dim2} has a simple equivalent form.
 
 \begin{lemma}\label{le:condition for blaschke}
 	If $ B $ is a finite Blaschke product, then~\eqref{eq:assumption dim2} is satisfied if and only if 
 	\begin{equation}\label{eq:assumption dim}
 		\dim K_B\le \dim K_\theta.
 	\end{equation}
  
 \end{lemma}
 
 \begin{proof}
 	Indeed, first assume that \eqref{eq:assumption dim} is satisfied, and let $f\in \ker T_{\theta \overline B}$. Then  $\theta f\in\ker T_{\overline B}=K_B$, whence $f=T_{\overline \theta}(\theta f)\in T_{\overline \theta}K_B\subset K_B$. If $f\not\equiv 0$, then $\theta=\frac{\theta f}{f}$ is a quotient of two polynomials of degree at most $\mathrm{deg} B-1$, which contradicts assumption \eqref{eq:assumption dim}.
 	
 	 Suppose now that $ \dim K_B>\dim K_\theta $. Then $ \theta H^2 $ has finite codimension in $ H^2 $ strictly smaller than $ \dim K_B $, whence $ \theta H^2\cap K_B\not=\{0\} $. Applying~\eqref{eq:small observation} in case $ v=\theta $, it follows that $ \ker T_{\theta \overline B}\not= \{0\}$, contradicting~\eqref{eq:assumption dim2}. 
 \end{proof}
 
Condition~\eqref{eq:assumption dim} is precisely the one considered in~\cite{Gu} and~\cite{LYL}. To discuss this case, we need one more elementary lemma.
	
	\begin{lemma}\label{le:poly blaschke}
		Suppose $ h_1, h_2 $ are two polynomials of degree at most $ k-1 $  and 
		\begin{equation}\label{eq:|h_1|=|h_2|}
		|h_1|=|h_2| \text{ a.e. on } \bbT.
		\end{equation}
Then, 		
		\[
		\frac{h_2}{h_1}=\frac{B_2}{B_1},
		\]
		where $ B_i $ are Blaschke products with $ \deg B_1+\deg B_2\le k-1 $.
	\end{lemma}

	\begin{proof}
		First, a general remark. Suppose that $h$ is a polynomial and write $ h(z)=z^p g(z) $, with $p\in\mathbb N\cup\{0\} $ and  $g(0)\not=0 $. Denote the roots (counting with multiplicities) of $ g$ by  $ \alpha_1, \dots, \alpha_\ell $. Then, $ h^o $, the outer part of $h$, is a polynomial of degree $ \deg g$, which  
		has roots $ Z_o(h)\cup Z_i(h)$, where $Z_o(h) := \{\alpha_i: |\alpha_i|\ge 1\}$ and $ Z_i(h):=
		\{ 1/\bar\alpha_i :0<|\alpha_i|< 1 \} $.

We may assume that $ h_1, h_2 $ have no common roots (otherwise we cancel them). It also follows then that $ h_1 $ and $ h_2 $ have no roots on $ \bbT $ (since this would be a common root by~\eqref{eq:|h_1|=|h_2|}  ). Also, only one of them may have the root 0; suppose it is $ h_1 $, and write, as above, $ h_1(z)=z^p g_1(z) $, with $ g_1(0)\not=0 $.

		 Assumption~\eqref{eq:|h_1|=|h_2|} implies that the outer parts of $ g_1 $ and $ h_2 $ coincide. 
			Since $ g_1 $ and $ h_2 $ have no common roots, but $ g_1{}^o=h_2^o $, we must have $ Z_o(g_1)=Z_i(h_2) $ and $ Z_i(g_1)=Z_o(h_2) $. Then we can write $ h_2/h_1=B_2/B_1 $, with
		\[
			B_1=z^p\prod_{\alpha_i\in Z^\sharp_i(h_1)} \phi_{\alpha_i}, \quad
		B_2=\prod_{\alpha_i\in Z^\sharp_i(h_2)} \phi_{\alpha_i},
		\]
where $Z^\sharp_i(p)=\{\alpha_i:p(\alpha_i)=0, 0<|\alpha_i|<1\}=\{1/\bar\alpha_i:\alpha_i\in Z_i(p)\}$. 
	Since we have 
	\[
	\deg B_1+\deg B_2=p+|Z_i(g_1)|+|Z_i(h_2)|=
p+	|Z_i(g_1)|+|Z_o(g_1)|\le k-1,
	\]
	the lemma is proved.
	\end{proof}

The next theorem generalizes~\cite[Theorem 1.4]{LYL}.
	
	\begin{theorem}\label{th:main one-dimensional} Suppose $ B $ is a finite Blaschke product, while $ \theta $ is an inner function with $ \deg \theta\ge \deg B $. Then
		the operator $ A^\theta_B $ has a reducing subspace such that the restriction has one-dimensional defects if and only if 
		\begin{equation}\label{eq2:main one-dim}
		\theta=\frac{B_2}{B_1}(u\circ B),
		\end{equation}
		where $ u $ is a non constant inner function, while $ B_1, B_2 $ are finite Blaschke products with $ \deg B_1+\deg B_2\le \deg B-1 $.		
	\end{theorem}
	
	\begin{proof}
We apply to this case Corollary~\ref{cor-reducing-general-case} (ii).  The existence of the required reducing subspace is then equivalent to the existence of $ h_1, h_2\in K_B $ and an inner function $ u $, such that 		
\begin{equation}\label{eq2:theta h_1=h_2 u}
			\theta h_1=h_2(u\circ B).
\end{equation}
		By~\eqref{eq:elements in K_B}, it is equivalent to assume in this equality that $ h_i $ are polynomials of degree $ \le k-1 $, where $k=\deg B$. 
		Taking absolute values, we obtain, since $ \theta $ and $ u\circ B $ are inner, that $ |h_1|=|h_2| $ on $ \bbT $.
		We may then apply Lemma~\ref{le:poly blaschke} to obtain the desired formula~\eqref{eq2:main one-dim}.

	 The converse is immediate, since~\eqref{eq2:main one-dim} implies~\eqref{eq2:theta h_1=h_2 u}, with the degrees of $ h_1 $ and $ h_2 $ at most $ k-1 $. If we further write $ g_i(z)=\frac{h_i(z)}{\prod_{i=1}^k (1-\bar{w_i}z)}  $, we obtain
 \[
 \theta g_1= g_2(u\circ B).
 \]
 Since $ g_i\in K_B $, this
  is equivalent, by Corollary~\ref{cor-reducing-general-case}, to the existence of the required reducing subspace.
	\end{proof}

The condition becomes simpler if $ \theta $ is singular.

\begin{theorem}	
Let $\theta$ be a singular inner function and let $B$ be a finite Blaschke product. Then the operator $ A^\theta_B $ has a reducing subspace such that the restriction has one-dimensional defects if and only if 
		\begin{equation}\label{eq3:main one-dim}
		\theta=S\circ B,
		\end{equation}
for some singular inner function $S$.
\end{theorem}
\begin{proof}
According to Theorem~\ref{th:main one-dimensional}, it is sufficient to prove that \eqref{eq3:main one-dim} and \eqref{eq2:main one-dim} are equivalent. The implication $\eqref{eq3:main one-dim}\Longrightarrow \eqref{eq2:main one-dim}$ is clear. Assume now that \eqref{eq2:main one-dim} is satisfied, that is we can write 
\[
B_1\theta=B_2(u\circ B),
\]
 where $B_1$ and $B_2$ are finite Blaschke products with $ \deg B_1+\deg B_2\le N-1$ and $N=\deg B$. 
 
 Since $ \theta $ is singular, $ B_2 $ must be a factor of $ B_1 $ and may be canceled. So we may assume $ B_2=1 $, or $B_1\theta=u\circ B  $, where $ \deg B_1\le N-1 $.
 
 Write then $u=B_3 S$, where $B_3$ is a Blaschke product and $S$ is the singular part of $u$. Thus we have
\[
B_1\theta =(B_3\circ B) (S\circ B).
\]
We have $ \deg (B_3\circ B)=\deg B_3 \deg B $; so, if $ B_3 $ is not constant, then
\[
\deg(B_3\circ B)\ge \deg B=N>\deg B_1.
\]
The contradiction obtained implies that $ B_3 $ is constant, and so
\[
B_1\theta =  S\circ B.
\]
Since the right hand side is singular, it follows that $ B_1 $ is constant, which proves the theorem.
\end{proof}

	\section{The case $ B(z)=z^N $}\label{se:monomials}
	
	The case $ B(z)=z^N $ is investigated at length in~\cite{Gu}. In particular, the characteristic function of $ A^\theta_{z^N} $ is computed; let us show how Gu's result follows from Theorem~\ref{th:identification of Phi} above. 

	We use the canonical basis of $ K_B $   formed by $ 1, z, \dots z^{N-1} $. To obtain the matrix of $ A^B_{\frac{\theta}{1-\lambda\overline{ B}}} $, consider first  $ A^B_{\frac{z^n}{1-\lambda\overline{ B}}} $. We have
		\[
		\frac{z^n}{1-\lambda\overline{ B}}=\sum_{j=0}^{\infty}\lambda^j z^{n-jN}
		=\sum_{j=0}^{\infty}\lambda^j z^{N(n'-j)+m}  ,
		\]
		where $ n=Nn'+m $, with $ 0\le m\le N-1 $. Since $ A^B_{z^p} $ is nonzero only for $ -(N-1)\le p \le N-1 $, we have to consider in the above sum only two terms, corresponding to $ j=n' $ and $ j=n'+1 $. Thus
		\begin{equation*}
		A^B_{\frac{z^n}{1-\lambda\overline{ B}}} = A^B_{\lambda^{n'}z^m +\lambda^{n'+1}z^{m-N} }.
		\end{equation*}
		Its matrix with respect to the canonical basis is
		\begin{equation}\label{eq:A^B_ for monomials}
		A^B_{\lambda^{n'}z^m +\lambda^{n'+1}z^{m-N} }=
		\begin{pmatrix}
		\ddots &  \ddots &\lambda^{n'+1}&\ddots &\ddots  \\
		\ddots & \ddots &   \ddots &\lambda^{n'+1}&\ddots\\
		\lambda^{n'}&\ddots&\ddots& \ddots&\ddots\\
		\ddots &\lambda^{n'}& \ddots& \ddots&\ddots\\
		\ddots&\ddots&\ddots&\ddots&\ddots 
		\end{pmatrix},
		\end{equation}	
		with two nonzero constant diagonals (one in case $ m=0 $), corresponding to entries $ a_{ij} $ with $ i-j=m $ or $ i-j=m-N $.
		
		Therefore, if
		we decompose
		\[
		\theta(z)=\theta_0(z^N)+z\theta_1(z^N)+\dots+z^{N-1}\theta_{N-1}(z^N), 
		\]
		then 
		\begin{equation}\label{eq:char function in example}
		A^B_{\frac{\theta}{1-\lambda\overline{ B}}}=
		\begin{pmatrix}
		\theta_0(\lambda) & \lambda \theta_{N-1}(\lambda) &\lambda \theta_{N-2}(\lambda)&\dots & \lambda \theta_1(\lambda) \\
		\theta_1(\lambda)& \theta_0(\lambda)&  \lambda \theta_{N-1}(\lambda)&\dots&\lambda\theta_2(\lambda)\\
		\theta_2(\lambda)& \theta_1(\lambda)&\theta_0(\lambda)&\dots& \lambda\theta_3(\lambda)\\
		\ddots&\ddots&\ddots&\ddots&\ddots\\
		\theta_{N-1}(\lambda)&\theta_{N-2}(\lambda)& \theta_{N-3}(\lambda)&\dots&\theta_0(\lambda)
		\end{pmatrix}.
		\end{equation}
	 This is precisely the form given by~\cite[Theorem 2.4]{Gu}.
		
	 In the sequel we will solve a conjecture about $ A^\theta_{z^N} $ left open in~\cite{Gu}. This appears as Conjecture~3.5 therein, and has the following statement.

 \begin{conjecture} \label{conj A} Suppose $ B(z)=z^N $. Then the following are equivalent:
 	\begin{itemize}
 		\item[(i)]  	$ A_B^\theta $ has a reducing subspace such that the restriction has one-dimensional defects.
 		
 		\item[(ii)] $ \theta(z)=b(z)u(z^N) $ for some inner function $ u $, while either $ b\equiv 1 $  or  
 			\begin{equation}\label{eq:form of b 2}
 			b(z)=\prod_{i=1}^{l}\psi_{\alpha_i, J_i},
 		\end{equation}
 		where $ l\le N-1 $, $ J_i\subset \{0, \dots, N-1\} $, and $ \psi_{\alpha, J} $ is defined by
 		\begin{equation}\label{eq:psi}
 			\psi_{\alpha, J}(z)=\prod_{i\in J} \phi_{\omega^i\alpha}(z).
 		\end{equation}
 	\end{itemize}

 \end{conjecture}
 
	 \cite[Theorem 3.4]{Gu} shows that (i)$ \implies $(ii), while (ii)$ \implies $(i) is proved only for $ N=3 $ in \cite[Section~5]{Gu}.

%
%

\begin{theorem}\label{th:conjecture A is false}
Conjecture~\ref{conj A}   is false for $ N=4 $.	
\end{theorem}

\begin{proof}
	Take two different nonzero values $ \alpha,\beta\in \bbD $, and define
	\[
	\theta(z)= \frac{(z^2-\alpha^2)(z^2-\beta^2)}{(1-\bar\alpha^2z^2)(1-\bar\beta^2z^2)}.
	\]
	We have then 
	\[
	\theta(z)=\psi_{\alpha, J}\psi_{\beta, J}
	\]
	with $ J=\{0, 2\}\subset\{0,1,2,3\} $, so $ \theta $ satisfies condition (ii) of Conjecture~\ref{conj A}.
	
	On the other hand, if $ \theta $ would satisfy condition (i), it would follow by Theorem~\ref{th:main one-dimensional} that one should have
	\begin{equation}\label{eq:degrees}
	B_2(z)\theta(z)= B_1(z)u(z^4),
	\end{equation}
	with $ u $ inner and $ B_1, B_2 $ finite Blaschke products with $ \deg B_1+\deg B_2\le 3 $.
	
	Obviously $ u $ has also to be a finite Blaschke product. Equating the degrees in both sides yields
	\[
	\deg B_1+4=\deg B_2+4\deg u.
	\]
	First, $ \deg B_1=3 $ would imply $ \deg B_2=0 $, so $ 7=4\deg u $: a contradiction. So the degree of the left hand side of~\eqref{eq:degrees} is between 4 and 6, which implies $ \deg u=1 $. Again equating the degrees yields $ \deg B_1=\deg B_2=0\ \text{or}\ 1 $.
	
	Now $ u(z^4) $ has either the root 0 of multiplicity 4, or four distinct roots. Both possibilities are incompatible with the fact that the left hand side of~\eqref{eq:degrees} has either 2 or three roots. We have obtained the desired contradiction, so $ \theta $ does not satisfy (i) of Conjecture~\ref{conj A}.
\end{proof}

In fact, we may replace Conjecture~\ref{conj A} with a precise result.
We will need the next lemma, also proved in~\cite[Theorem 3.4]{Gu}. 

\begin{lemma}\label{le:whole blaschke}
	If $ \alpha\in\bbD $, then
	\[
	\phi_{\alpha^N} (z^N)=\prod_{i=0}^{N-1} \phi_{\omega^i\alpha}(z).
	\]
\end{lemma}

\begin{theorem}\label{th:connection with Gu}
	Suppose $ B(z)=z^N $. Then the following are equivalent:
	\begin{itemize}
		\item[(i)]  	$ A_B^\theta $ has a reducing subspace such that the restriction has one-dimensional defects.
		
		\item[(ii)] $ \theta(z)=b(z)u(z^N) $ for some inner function $ u $, while $ b $ is either 1 or a finite Blaschke product given by~\eqref{eq:form of b 2},
	 where $ l\le N-1 $, $ J_i\subset \{0, \dots, N-1\} $, $ \psi_{\alpha, J} $ are defined by~\eqref{eq:psi}, and, moreover,
	 \begin{equation}\label{eq:condition on psi}
	 \sum\limits_{i=1}^{l} \min\{|J_i|, N-|J_i|\}\le N-1.
	 \end{equation}
	 
	\end{itemize}

\end{theorem}

\begin{proof}  (i)$ \implies $(ii).
	From Theorem~\ref{th:main one-dimensional} we know that $ \theta $ is given by ~\eqref{eq2:main one-dim}, where $ B_1 $ and $ B_2 $ have no common roots. We may denote the roots of $ B_1 $ (counting multiplicities) by \[ \{\alpha_1^1, \dots, \alpha^1_{s_1}; \alpha^2_1,\dots, \alpha^2_{s_2}; \dots ;\alpha^p_1,\dots, \alpha^p_{s_p} \},
	\]
where, for each $ i=1, \dots, p $, the values $ \alpha^i_1, \dots , \alpha^i_{s_i} $ are all distinct, and
\[
(\alpha^i_1)^N=\dots=(\alpha^i_{s_i})^N.
\]
Similarly, we denote the roots of $ B_2 $ by
\[ \{\beta_1^1, \dots, \beta^1_{r_1}; \beta^2_1,\dots, \beta^2_{r_2}; \dots ;\beta^q_1,\dots, \beta^q_{r_q} \},
\]
where, for each $ i=1, \dots, q $, the values $ \beta^i_1, \dots , \beta^i_{r_i} $ are all distinct, and
\[
(\beta^i_1)^N=\dots=(\beta^i_{s_i})^N.
\]
Note that the condition $ \deg B_1+\deg B_2\le N-1 $ is transcribed as
\begin{equation}\label{eq:s_i r_i}
s_1+\dots+s_p+r_1+\dots +r_q\le N-1.
\end{equation}
In particular, $ p+q\le N-1 $.

Now, it is easy to see that, for each $ i=1, \dots, q $, the Blaschke product
\[
\phi_{\beta^i_1}\dots \phi_{\beta^i_{r_i}}
\]
is equal to $ \psi_{\beta^i_1, J_i} $ for some $ J_i\subset \{0, \dots, N-1\} $. So
\begin{equation}\label{eq:B_2 psi}
B_2=\prod_{i=1}^{q}  \psi_{\beta^i_1, J_i}.
\end{equation}

The matter is  more subtle as concerns $ B_1 $: it appears at the denominator, which we do not want. We have, similarly,
\begin{equation}\label{eq:B_1 psi 1}
B_1=\prod_{i=1}^{p}  \psi_{\alpha^i_1, J'_i}
\end{equation}
for some $ J'_i\subset \{0, \dots, N-1\} $.

The factor $ \phi_{\alpha^1_1}(z) $ must be canceled by a factor in $ u(z^N) $, so $\alpha^1_1  $ must be a root of $ u(z^N) $. But then $ u(z^N) $ must also have the roots $ \omega^j \alpha^1_1 $ for $ j=1, \dots, N-1 $, and so  
\[
u(z^N)=\prod_{j=0}^{N-1} \phi_{\omega^j\alpha^1_1}(z) u_1(z^N).
\]
Since
\[
\frac{\phi_{\omega^j\alpha^1_1}(z)}{\psi_{\alpha^1_1, J'_1}}=\psi_{\alpha^1_1, J''_1}
\]
with $ J''_1=\{0, \dots, N-1\}\setminus J'_1 $, we have
\[
\frac{u(z^N)}{\psi_{\alpha^1_1, J'_1}}=\psi_{\alpha^1_1, J''_1}u_1(z^N).
\]
We may continue the argument (or use an appropriate induction) to obtain
\begin{equation}\label{eq:B_1 psi modified}
\frac{u(z^N)}{B_1(z)}=\prod_{i=1}^{p}  \psi_{\alpha^i_1, J''_i}u'(z^N)
\end{equation}
for an inner function $ u' $, where $ J''_i=\{0, \dots, N-1\}\setminus J'_i $. From~\eqref{eq2:main one-dim},~\eqref{eq:B_2 psi}, and~\eqref{eq:B_1 psi modified} it follows that
\[
\theta(z)=\prod_{i=1}^{q}  \psi_{\beta^i_1, J_i}\prod_{i=1}^{p}  \psi_{\alpha^i_1, J''_i}\  u'(z^N).
\]
This is exactly the form given by~\eqref{eq:form of b 2}.   Moreover $ \min\{|J_i|, N-|J_i|\}\le r_i  $ and $ \min\{|J_i''|, N-|J_i''|\}\le s_i  $, so~\eqref{eq:s_i r_i} implies~\eqref{eq:condition on psi}.

(ii)$ \implies $(i). Suppose $ b(z) $ is given by~\eqref{eq:form of b 2}, with~\eqref{eq:condition on psi} satisfied. Define  
\[
B_2=\prod_{\min\{|J_i|, N-|J_i|\}=|J_i|} \psi_{\alpha_i, J_i}
\]
and
\[
B_1=\prod_{\min\{|J_i|, N-|J_i|\}=N-|J_i|} \psi_{\alpha_i, N\setminus J_i}.
\]
Then
\[
\theta(z)=\frac{B_2(z)}{B_1(z)}u_1(z^N),
\]
where
\[
u_1(z)=u(z)\prod_{\min\{|J_i|, N-|J_i|\}=N-|J_i|}\phi_{\alpha_i^N}(z);
\]
note that we have used Lemma~\ref{le:whole blaschke}.
\end{proof}

%
%
%
%
%

\end{document}